\documentclass{amsart}
\usepackage{amscd}
\usepackage{enumerate,amssymb,amsmath,graphics,epsfig}
\usepackage[colorlinks]{hyperref}
\usepackage{color}
\usepackage{tcolorbox}
\usepackage{mathrsfs}

\newcommand\RE{\mathbb{R}}
\newcommand\NA{\mathbb {N}}
\renewcommand\div{\mathop{\rm{div}}\nolimits}

\newcommand\FF{\boldsymbol{F}}

\newcommand\F{\mathcal{F}}

\newcommand\HH{\mathbf{H}}
\newcommand\Hrot{\HH(\rot;\Omega)}
\newcommand\Hcurl{\HH(\curl;\Omega)}

\newcommand\Hdiv{\HH(\div;\Omega)}

\newcommand\Hocurl{\HH_0(\curl;\Omega)}
\newcommand\Hodivo{\HH_0(\div^0;\Omega)}

\newcommand\Huo{\H^1_0(\Omega)}

\newcommand\I{\mathrm{I}}

\newcommand\Ph{\boldsymbol{P}_{\!h}}
\newcommand\PH{\boldsymbol{P}_{\!H}}

\newcommand\W{\boldsymbol{\mathcal{W}}}
\newcommand\bbeta{\boldsymbol{r}}
\newcommand\grad{\operatorname{\boldsymbol{\nabla}}}
\newcommand\curl{\operatorname{\mathbf{curl}}}
\newcommand\rot{\operatorname{curl}}

\newcommand\hsH{\widehat{\ssigma}_H}
\newcommand\hpH{\widehat{\pp}_H}
\newcommand\tHi{\widetilde{\ssigma}_{H,i}}
\newcommand\tH{\widetilde{\ssigma}_H}
\newcommand\lh{\lambda_h}
\newcommand\lH{\lambda_H}
\newcommand\n{\boldsymbol{n}}
\newcommand\omh{\om_h}
\newcommand\om{\omega^2}
\newcommand\pphi{\boldsymbol{w}}
\newcommand\ppsi{\boldsymbol{\xi}}
\newcommand\sh{\ssigma_h}
\newcommand\sH{\ssigma_H}
\newcommand\ssigma{\boldsymbol{\sigma}}
\newcommand\cchi{\boldsymbol{\chi}}
\newcommand\xxi{\boldsymbol{\vartheta}}

\newcommand\T{\mathcal T}
\newcommand\tria{\mathcal{T}_h}

\newcommand\ttau{\boldsymbol{\tau}}

\newcommand\uu{\boldsymbol{u}}

\newcommand\vv{\boldsymbol{v}}
\newcommand\ww{\boldsymbol{\zeta}}

\newcommand\zz{\boldsymbol{z}}
\newcommand\hh{\boldsymbol{\mathfrak{h}}}
\renewcommand\ss{\boldsymbol{s}}
\newcommand\pp{\boldsymbol{p}}
\newcommand\qq{\boldsymbol{q}}
\newcommand\N{N}
\newcommand\Q{\boldsymbol{Q}}
\newcommand\Sh{\boldsymbol{\Sigma}_h}
\newcommand\SH{\boldsymbol{\Sigma}_H}
\renewcommand\H{\mathrm{H}}

\newcommand{\Ws}{\widetilde W}
\newcommand{\Wm}{W}
\renewcommand\dim{\operatorname{\mathrm{dim}}}
\renewcommand\div{\operatorname{\mathrm{div}}}

\newcommand\card{\mathrm{card}}
\newcommand\etas{\tilde\eta}
\newcommand\etam{\eta}
\newcommand*{\jump}[1]{\lbrack\hspace{-2pt}\lbrack%
#1\rbrack\hspace{-2pt}\rbrack}
\newcommand\Span{\mathrm{span}}
\newcommand\errsh{\|\ssigma-\sh\|_0}
\newcommand\errph{\|\pp-\pp_h\|_0}
\newcommand\errsH{\|\ssigma-\ssigma_H\|_0}
\newcommand\errpH{\|\pp-\pp_H\|_0}
\newcommand\errsl{\|\ssigma_\ell-\ssigma_{\ell+1}\|_0}
\newcommand\errpl{\|\pp_\ell-\pp_{\ell+1}\|_0}
\newcommand\errshH{\|\sh-\ssigma_H\|_0}
\newcommand\errphH{\|\pp_h-\pp_H\|_0}
\newcommand\rhosc{\rho_{\mathrm{sc}}}
\newcommand\g{\mathbf{g}}
\newcommand\sg{\ssigma_g}
\newcommand\pg{\pp_g}
\newcommand\sgh{\ssigma_{g,h}}
\newcommand\pgh{\pp_{g,h}}
\newcommand\fortinH{\boldsymbol{\Pi}^F_H}
\newcommand\hx{\boldsymbol{\mathcal{P}}_H}
\newcommand\schoberl{\boldsymbol{\mathcal{S}}_H}
\newcommand\harm{\mathcal{H}}
\newcommand\nome{\mathscr{E}}

\DeclareSymbolFont{sfletters}{OML}{cmbrm}{m}{it}
\DeclareMathSymbol{\sfsigma}{\mathord}{sfletters}{"1B}

\theoremstyle{plain}
\newtheorem{thm}{Theorem}
\newtheorem{proposition}[thm]{Proposition}
\newtheorem{lemma}[thm]{Lemma}

\theoremstyle{remark}

\newtheorem{remark}{Remark}

\newtheorem{property}{Property}

\begin{document}

\title[AFEM for Maxwell's eigenvalues]
{Adaptive finite element method for the Maxwell eigenvalue problem}
\author{Daniele Boffi}
\address{Dipartimento di Matematica ``F. Casorati'', Universit\`a di Pavia,
Italy and Department of Mathematics and System Analysis, Aalto University,
Finland}
\email{daniele.boffi@unipv.it}
\urladdr{http://www-dimat.unipv.it/boffi/}
\author{Lucia Gastaldi}
\address{DICATAM, Universit\`a di Brescia, Italy}
\email{lucia.gastaldi@unibs.it}
\urladdr{http://lucia-gastaldi.unibs.it}
\subjclass{65N30, 65N25, 35Q61, 65N50}
\keywords{edge finite elements, eigenvalue problem, Maxwell's equations,
adaptive finite element method}

\begin{abstract}
In this paper we prove the optimal convergence of a standard adaptive scheme
based on edge finite elements for the approximation of the solutions of the
eigenvalue problem associated with Maxwell's equations. The proof uses the
known equivalence of the problem of interest with a mixed eigenvalue problem.
\end{abstract}
\maketitle

\section{Introduction}

In this paper we present an adaptive scheme, based on standard three
dimensional edge elements, for the approximation of the Maxwell eigenvalue
problem and analyze its convergence.

A posteriori error estimates for Maxwell's equations have been studied by
several authors for the source problem (see, in
particular~\cite{monk98,beck,nic0,cina,nic1,nic2,schoberl,nicc1,burg1,nicc2,nicc3,burg2,nicc4,nicc5} and the
references therein). The eigenvalue problem has been studied only recently
in~\cite{BGRI1,BGRI2} where residual type error indicators are considered and
proved to be equivalent to the actual error in the standard framework of
efficiency and reliability. The analysis relies on the classical equivalence
with a mixed variational formulation~\cite{bfgp}. The numerical results
presented in~\cite{BGRI2} confirm that the adaptive scheme driven by our
error indicator converges in three dimensions with optimal rate with respect
to the number of degrees of freedom.

In~\cite{dietmar} it was presented the first convergence analysis for an
adaptive scheme applied to the Laplace eigenvalue problem in mixed form. The
main tools for the analysis originate from various papers related to adaptive
finite elements, in particular from~\cite{ckns,gmz}.
Thanks to the well-known isomorphism
between the spaces $\Hrot$ and $\Hdiv$ in two space dimensions, the result for
the Laplacian implies the convergence of the 2D adaptive scheme for Maxwell's
eigenproblem: actually, the isomorphism carries over to the corresponding
mixed formulation as well as to the error indicators.
In this paper we extend the results of~\cite{dietmar} to the mixed formulation
associated with Maxwell's eigenproblem in three dimensions; as we will notice,
such extension is not trivial: several technical details have to be checked
and suitably designed interpolation operators are used to complete the
analysis. Useful results in this direction are reported
in~\cite{schoberl,zcswx}.

It is well understood that the convergence analysis of the adaptive scheme
for eigenvalue problems has to consider multiple eigenvalues and clusters of
eigenvalues in order to prevent subtoptimal convergence. In particular,
degeneracy of the convergence may be observed when the error indicator is
computed by taking into account only a subset of the discrete eigenmodes
approximating the eigensolutions we are interested in (multiple or belonging
to a cluster)~\cite{giani,gallistl,duran}.

Starting from this remark, the analysis performed in~\cite{dietmar} has been
carried on for generic clusters of eigenvalues. This approach has the
inconvenience of adding a heavy notation dealing with deep technicalities. For
this reason, we decided in this paper to develop our theory in the case of
simple eigenvalues. We believe that the presentation in the case of a simple
eigenvalue highlights better the novelties with respect to the previous
results for the mixed Laplacian, that would rather be hidden by the technical
machinery related to clusters of eigenvalues. Nevertheless, the general case
can be dealt with by using similar arguments as in~\cite{dietmar}.

In Section~\ref{se:maxwell} we recall Maxwell's eigenvalue problem, its
standard variational formulation, and the equivalent mixed formulation,
together with their finite element discretizations. Section~\ref{se:adaptive}
defines our error indicator and describes the adaptive scheme. Reliability and
efficiency from~\cite{BGRI2} are recalled and the theory concerning the
convergence of the adaptive method is described. The auxiliary results needed
for the convergence proof are collected in Section~\ref{se:auxiliary}. These
include in particular discrete reliability, quasi-orthogonality, and
contraction property.

\section{Maxwell's eigenvalue problem and its finite element discretization}
\label{se:maxwell}

In this paper we deal with the well known eigenvalue problem associated with
the Maxwell equations (see, for instance,~\cite{hiptmair,monk,acta}).

Given a polyhedral domain $\Omega$, after eliminating the magnetic field, the
problem reads: find $\omega\in\RE$ and $\uu\ne0$ such that
\begin{equation}
\aligned
&\curl(\mu^{-1}\curl\uu)=\omega^2\varepsilon\uu&&\text{in }\Omega\\
&\div(\varepsilon\uu)=0&&\text{in }\Omega\\
&\uu\times\n=0&&\text{on }\partial\Omega,
\endaligned
\label{eq:strong}
\end{equation}
where $\uu$ represents the electric field, $\mu$ and $\varepsilon$ the magnetic
permittivity and electric permeability, respectively, and $\n$ is the outward
unit normal vector to $\partial\Omega$, the boundary of $\Omega$. For general
inhomogeneous, anisotropic materials $\mu$ and $\varepsilon$ are $3$-by-$3$
positive definite and bounded matrix functions.
We are considering for simplicity the case when $\Omega$ is simply connected:
more general situations will be described in Remark~\ref{re:multiply}.

A standard variational formulation of our eigenvalue problem is obtained by
considering the functional space $\Hocurl$ of vector fields in $L^2(\Omega)^3$
with $\curl$ in $L^2(\Omega)^3$ and vanishing tangential component along
$\partial\Omega$. The formulation reads: find $\omega\in\RE$ with $\omega>0$
and $\uu\in\Hocurl$ with $\uu\ne0$ such that
\begin{equation}
(\mu^{-1}\curl\uu,\curl\vv)=\omega^2(\varepsilon\uu,\vv)
\quad\forall\vv\in\Hocurl.
\label{eq:var}
\end{equation}
It is well known, in particular, that the condition $\omega^2\ne0$ is
equivalent to the divergence condition $\div(\varepsilon\uu)=0$ due to the
Helmholtz decomposition (see also Remark~\ref{re:Nh}). We assume that the
domain $\Omega$ and the coefficients $\varepsilon$, $\mu$ are such that the
problem is associated with a compact solution operator.
The eigenvalues can then be numbered in an increasing order as follows:
\[
0<\omega_1\le\omega_2\le\dots\le\omega_j\le\dots,
\]
where the same eigenvalue is repeated as many times as its algebraic
multiplicity. The associated eigenfunctions are denoted by $\uu_j$ and
normalized according to the $L^2$ norm, that is
$\|\varepsilon^{1/2}\uu_j\|_0=1$.

A powerful tool for the analysis of this problem is a mixed formulation
introduced in~\cite{bfgp}. With the notation $\ssigma=\omega\uu$,
$\pp=-\mu^{-1/2}\curl\uu/\omega$, and $\lambda=\omega^2$, the variational
formulation~\eqref{eq:var} is equivalent to the following mixed eigenproblem:
find $\lambda\in\RE$ and $(\ssigma,\pp)\in\Hocurl\times\Q$ with $\pp\ne0$ such
that
\begin{equation}
\aligned
&(\varepsilon\ssigma,\ttau)+(\mu^{-1/2}\curl\ttau,\pp)=0
&&\forall\ttau\in\Hocurl\\
&(\mu^{-1/2}\curl\ssigma,\qq)=-\lambda(\pp,\qq)&&\forall\qq\in\Q,
\endaligned
\label{eq:varmixed}
\end{equation}
where $\Q=\mu^{-1/2}\curl(\Hocurl)$.

The eigenvalues of~\eqref{eq:varmixed} are denoted by
\[
0\le\lambda_1\le\lambda_2\le\dots\le\lambda_j\le\dots.
\]
Given $j$, we use the notation $\pp_j=-\mu^{-1/2}\curl\uu_j/\omega_j$ and
$\ssigma_j=\omega_j\uu_j$ with $\lambda_j=\omega_j^2$, so that
$(\lambda_j,\ssigma_j,\pp_j)$ solves~\eqref{eq:varmixed} and the following
normalization holds true for the eigenfunction: $\|\pp_j\|_0=1$.

The finite element approximation of~\eqref{eq:var} is usually performed with
edge elements. Given a tetrahedral decomposition of $\Omega$, we consider
N\'ed\'elec edge elements introduced in~\cite{nedelec1,nedelec2}. More general
families of finite elements could be considered in the spirit
of~\cite{periodic-table}. More precisely, the general situation can be
described by adopting the following standard notation related to de Rham
complex:
\begin{equation}
\minCDarrowwidth25pt
\begin{CD}
0@>>>\Huo@>\grad>>\Hocurl@>\curl>>\Hodiv@>\div>>L^2(\Omega)@>>>\RE\\
& & @VVV @VVV @VVV @VVV & &\\
0@>>>\N_h@>\grad>>\Sh@>\curl>>\FF_h@>\div>>DG_h@>>>\RE.
\end{CD}
\label{eq:derham}
\end{equation}

In the case when $\Sh$ is a sequence of tetrahedral edge finite elements the
remaining finite element spaces will be composed by nodal Lagrange elements
$\N_h$, face elements $\FF_h$, and discontinuous elements $DG_h$,
respectively. The corresponding diagrams in the case of N\'ed\'elec elements
of the first and second family can be found in~(2.5.58) and~(2.5.59)
of~\cite{bbf}, respectively.

The discretization of~\eqref{eq:var} reads: find $\omega_h\in\RE$ with
$\omega_h>0$ and $\uu_h\in\Sh$ with $\uu_h\ne0$ such that
\begin{equation}
(\mu^{-1}\curl\uu_h,\curl\vv)=\omega_h^2(\varepsilon\uu_h,\vv)
\quad\forall\vv\in\Sh.
\label{eq:varh}
\end{equation}

The corresponding mixed formulation is: find
$\lh\in\RE$ and $(\sh,\pp_h)\in\Sh\times\Q_h$ with $\pp_h\ne0$ such
that
\begin{equation}
\aligned
&(\varepsilon\sh,\ttau)+(\mu^{-1/2}\curl\ttau,\pp_h)=0
&&\forall\ttau\in\Sh\\
&(\mu^{-1/2}\curl\sh,\qq)=-\lh(\pp_h,\qq)&&\forall\qq\in\Q_h,
\endaligned
\label{eq:varmixedh}
\end{equation}
where $\Q_h=\mu^{-1/2}\curl(\Sh)$. In particular, we have that $\Q_h$ is a
subspace of $\mu^{-1/2}\FF_h$ and it can be easily seen that
$\mu^{-1/2}\curl\sh=-\lh\pp_h$.

Following~\cite[Th.~2.1]{bfgp}, the equivalence between~\eqref{eq:varh}
and~\eqref{eq:varmixedh} can be proved using the definition of $\Q_h$
and the identities $\sh=\omega_h\uu_h$,
$\pp_h=-\mu^{-1/2}\curl\uu_h/\omega_h$, and $\lh=\omega_h^2$.

With natural notation, we denote by
$0<\omega_{h,1}\le\omega_{h,2}\le\dots\le\omega_{h,N(h)}$ the eigenvalues
of~\eqref{eq:varh} and by
$0<\lambda_{h,1}\le\lambda_{h,2}\le\dots\le\lambda_{h,N(h)}$ those
of~\eqref{eq:varmixedh}. Analogously, the corresponding eigenfunctions are
denoted by $\uu_{h,j}$ and $(\ssigma_{h,j},\pp_{h,j})$, respectively
($j=1,\dots,N(h)$) with $\|\varepsilon^{1/2}\uu_{h,j}\|_0=\|\pp_{h,j}\|_0=1$.
The number of discrete frequencies, repeated according to their multiplicity,
is given by $N(h)=\dim\Q_h$. We discuss this fact in the next remark.
\begin{remark}
It is straightforward to check that the number of real eigenvalues of
problem~\eqref{eq:varmixedh} is equal to $N(h)=\dim\Q_h$.
Indeed, the matrix form of~\eqref{eq:varmixedh} is, with obvious notation,
\[
\left(
\begin{matrix}
\mathsf{A}&\mathsf{B}^\top\\
\mathsf{B}&\mathsf{0}
\end{matrix}
\right)
\left(
\begin{matrix}
\sfsigma\\\mathsf{p}
\end{matrix}
\right)
=
\lambda_h
\left(
\begin{matrix}
\mathsf{0}&\mathsf{0}\\
\mathsf{0}&-\mathsf{M}
\end{matrix}
\right)
\left(
\begin{matrix}
\sfsigma\\\mathsf{p}
\end{matrix}
\right).
\]
The number of real eigenvalues of this problem is equal to the size $N(h)$ of
the matrix $\mathsf{M}$, as it is evident by looking at its equivalent
formulation written in terms of the Schur complement
\[
\aligned
&\mathsf{B}\mathsf{A}^{-1}\mathsf{B}^\top\mathsf{p}=\lambda_h\mathsf{M}\mathsf{p}\\
&\sfsigma=-\mathsf{A}^{-1}\mathsf{B}^\top\mathsf{p}.
\endaligned
\]
The size of the matrix problem corresponding to~\eqref{eq:varh} is equal to
the dimension of the space $\Sh$. The Helmholtz decomposition in the case of
simply connected domains implies that
$\dim(\Sh)=\dim(\grad(\N_h))+N(h)$. Since the space
$\grad(\N_h)$ is the kernel of the $\curl$ operator, it follows that
the number of eigenvalues corresponding to $\omega_h>0$ is equal to $N(h)$.
For an additional discussion about this count when the domain is multiply
connected, the reader is referred to Remark~\ref{re:multiply}.

\label{re:Nh}
\end{remark}
\begin{remark}
It can be useful to recall that the mixed formulations~\eqref{eq:varmixed}
and~\eqref{eq:varmixedh} are not used for the definition of the method (nor
for its implementation), but are crucial ingredients for its analysis.
\label{re:equiv}
\end{remark}

\begin{remark}
It is well known that if the domain is not topologically trivial, then the
first row of the diagram presented in Equation~\eqref{eq:derham} is not an
exact sequence. More precisely, the following space of \emph{harmonic forms}
plays an important role
\[
\harm=\{\hh\in\Hocurl:\curl\hh=0,\ \div(\varepsilon\hh)=0\text{ in }\Omega\}
\]
and corresponds to the one form cohomology of the de Rham complex. The
Helmholtz decomposition in this case has the following form:
\[
L^2(\Omega)^3=\grad(\Huo)\oplus\harm\oplus\varepsilon^{-1}\curl(\Hcurl),
\]
where the three components of the decomposition are $\varepsilon$-orthogonal,
that is they are orthogonal with respect to the scalar product
$(\varepsilon\,\cdot,\cdot)$.

It turns out that in the general case the formulation~\eqref{eq:var} is not
the variational formulation of~\eqref{eq:strong} any more. Indeed, functions
in $\harm$ are eigenfunctions of~\eqref{eq:strong} with vanishing frequency.
In this case, if we are not interested in the approximation of the space of
harmonic functions $\harm$, we can disregard the zero frequency and use
formulations~\eqref{eq:var} and~\eqref{eq:varmixed} for the analysis of the
rest of the spectrum. The approximation of harmonic functions is out of the
aims of this paper. We point the reader to possible approaches for the
approximation of $\harm$: a direct discretization of the space has been
proposed in~\cite{abgv}; an adaptive algorithm has been presented
in~\cite{demlow}; another indirect approach may be the use of the following
alternative mixed formulation known as \emph{Kikuchi formulation}
(see~\cite{kik,bbab}): find $\lambda\in\RE$ such that for $\uu\in\Hocurl$ and
$p\in\Huo$, with $\uu\ne0$, it holds
\[
\aligned
&(\mu^{-1}\curl\uu,\curl\vv)+(\grad p,\varepsilon\vv)=
\lambda(\varepsilon\uu,\vv)&&\forall\vv\in\Hocurl\\
&(\grad q,\varepsilon\uu)=0&&\forall q\in\Huo.
\endaligned
\]
It is not difficult to see that any solution of the Kikuchi formulation
satisfies $p=0$ (take $\vv=\grad p$ in the first equation). Hence, it is
immediate to check that the Kikuchi formulation is equivalent to the
standard variational formulation~\eqref{eq:var} with the additional solution
$\lambda=0$ corresponding to $\uu\in\harm$.
\label{re:multiply}
\end{remark}

\section{Error indicator and adaptive method}
\label{se:adaptive}

We are going to study and analyze an adaptive finite element scheme in the
framework of~\cite{dorfler,dxz,ckns,gmz,dietmar}. The scheme is based on the
following local error indicator (see~\cite{BGRI2})
\begin{equation}
\begin{split}
\etas^2_K={}&
h^2_K\|\varepsilon\uu_h-\curl(\mu^{-1}\curl\uu_h)/\omega_h^2\|^2_{0,K}
+h^2_K\|\div(\varepsilon\uu_h)\|^2_{0,K}\\
&+\frac12\sum_{F\in\F_{\I}(K)}
\left[h_F\left\|\jump{\left(\mu^{-1}\curl\uu_h/\omh\right)\times\n}\right\|^2_{0,F}
+h_F\left\|\jump{\varepsilon\uu_h\cdot\n}\right\|^2_{0,F}\right],
\end{split}
\label{eq:ind1}
\end{equation}
where $K$ is an element of our triangulation $\T_h$, $\F_{\I}(K)$ is the set
of inner faces of $K$, $h_K$ and $h_F$ the diameters of $K$, and $F$,
respectively, and $\jump{\cdot}$ the jump across an inner face $F$.

Given a set of elements $\mathcal M$, we use the notation
\[
\etas(\mathcal{M})^2=\sum_{K\in\mathcal{M}}\etas^2_K
\]
and we write $\etas=\etas(\T_h)$ for the global error indicator when no
confusion arises. Moreover, a subscript $\kappa$ is used when $\etas_\kappa$
refers to the mesh $\T_\kappa$.

Given an initial mesh $\T_0$ and a bulk parameter $\theta\in\RE$, with
$0<\theta\le1$, we compute a sequence of meshes $\{\T_\ell\}$, solutions
$\{(\omega^2_\ell,\uu_\ell)\}$, and indicators $\{\etas(\T_\ell)\}$
according to the standard \textsf{solve/estimate/mark/refine}
strategy (see~\cite{dorfler}). In particular, at a given level $\ell$, the
marking step consists in choosing a minimal subset $\mathcal{M}_\ell$ of
$\T_\ell$ such that
\[
\theta\etas^2_\ell(\T_\ell)\le\etas^2_\ell(\mathcal{M}_\ell).
\]
The new mesh $\T_{\ell+1}$ is given by the smallest admissible refinement of
$\T_\ell$ satisfying $\mathcal{M}_\ell\cap\T_{\ell+1}=\emptyset$ according to
the rules defined in~\cite{bdd,stevenson}.

Considering the equivalence between the standard formulation~\eqref{eq:varh}
and the mixed formulation~\eqref{eq:varmixedh}, the local error indicator for
the mixed problem takes the following form:
\begin{equation}
\begin{split}
\etam^2_K={}&
h_K^2\|\varepsilon\sh+\curl(\mu^{-1/2}\pp_h)\|^2_{0,K}
+h_K^2\|\div(\varepsilon\sh)\|^2_{0,K}\\
&+\frac12\sum_{F\in\F_{\I}(K)}
\left(h_F\left\|\jump{(\mu^{-1/2}\pp_h)\times\n}\right\|^2_{0,F}
+h_F\left\|\jump{\varepsilon\sh\cdot\n}\right\|^2_{0,F}\right)
\end{split}
\label{eq:ind2}
\end{equation}

It is easy to check that the following relation between the two indicators
holds true
\[
\etas^2_K=\frac1{\lh}\etam^2_K\quad\forall K\in\T_h.
\]
In particular, the comments stated in Remark~\ref{re:equiv} can be extended to
the error indicators: our analysis will be performed by using the mixed
formulation~\eqref{eq:varmixedh} and the indicator~\eqref{eq:ind2} even if the
scheme is originally defined in terms of the standard
formulation~\eqref{eq:varh} and the indicator~\eqref{eq:ind1}.

In the rest of this section we present our main result in the case of an
eigenvalue of multiplicity one, since we believe that in this case it is
easier to describe the main arguments leading to the optimal convergence of
the adaptive scheme.
Moreover, for ease of notation, we assume from now on that $\varepsilon$ and
$\mu$ are scalar and $\varepsilon=\mu=1$. This assumption does not reduce the
relevance of our result: more general situations can be dealt with by adopting
similar arguments as in~\cite{bfgp} or~\cite{cfr}.

Let $\omega=\omega_j$ be a simple eigenvalue of~\eqref{eq:var} and
$\Ws=\Span\{\uu_j\}$ the associated one-dimensional eigenspace. Let
$\omega_\ell=\omega_{\ell,j}$ be the $j$-th discrete eigenvalue
of~\eqref{eq:varh} computed with the adaptive scheme on the mesh $\T_\ell$ and
$\Ws_\ell=\Span\{\uu_{\ell,j}\}$ the corresponding eigenspace.
The gap between $\Ws$ and $\Ws_\ell$ is measured by
\[
\delta(\Ws,\Ws_\ell)=
\sup_{\substack{\uu\in\Ws\\ \|\uu\|_{\curl}=1}}\inf_{\uu_\ell\in \Ws_\ell}
\|\uu-\uu_\ell\|_{\curl}.
\]

For the reader's convenience, we recall the reliability and efficiency
properties proved in~\cite{BGRI2}. As it is common for eigenvalue problems,
the efficiency property is not local in the sense that it relies on the
difference between $\omega$ and $\omega_h$ which is a global quantity.

\begin{proposition}

Let $(\omega,\uu)$ and $(\omega_h,\uu_h)$ be solutions of
Problems~\ref{eq:var} and~\ref{eq:varh}, respectively, such that the latter
approximates the former as $h$ goes to zero. Then, there exists $C$ such that
for $h$ small enough
\begin{description}
\item[Reliability]
\[
\|\uu-\uu_h\|_{\curl}\le C\etas\qquad
|\omega^2-\omega^2_h|\le C\etas^2.
\]
\item[Efficiency]
\[
\etas_K\le C\left(\left\|\uu-\uu_h\right\|_{0,K'}+
\left\|\curl(\uu-\uu_h)\right\|_{0,K'}+
h_K\left\|\om\uu-\omh\uu_h\right\|_{0,K'}\right),
\]
where $K'$ denotes the union of the elements sharing a face with $K$.
\end{description}
\end{proposition}

\begin{proof}
See Propositions~5 and~6 of~\cite{BGRI2}.
\end{proof}

The convergence of the adaptive scheme is usually described by making use of
the nonlinear approximation classes discussed in~\cite{bdd}. Denoting by
$\mathbb{T}(m)$ the set of admissible refinements of $\T_0$ whose cardinality
differs from that of the initial triangulation by less than $m$, the best
algebraic convergence rate $s\in(0,+\infty)$ for the approximation of
functions belonging to a space $\W$ is characterized in terms of the
following seminorm
\[
|\W|_{\mathcal{A}_s}=
\sup_{m\in\NA}m^s\inf_{\T\in\mathbb{T}(m)}\delta(\W,\boldsymbol{\Sigma}_\T),
\]
where $\boldsymbol{\Sigma}_\T$ is the edge finite element space on the mesh
$\T$.

The main result of our paper, stated in the next theorem, shows that if $\Ws$
has bounded $\mathcal{A}_s$-seminorm for some $s$, then the optimal
convergence order $s$ is obtained by the sequence of solutions constructed by
the adaptive procedure described above.
\begin{thm}
Provided the meshsize of the initial mesh $\T_0$ and the bulk parameter
$\theta$ are small enough, if the eigenspace satisfies
$|\Ws|_{\mathcal{A}_s}<\infty$, then the sequence of discrete eigenspaces
$\Ws_\ell$ computed on the mesh $\T_\ell$ fulfills the optimal estimate
\[
\delta(\Ws,\Ws_\ell)\le
C(\card(\T_\ell)-\card(\T_0))^{-s}|\Ws|_{\mathcal{A}_s}.
\]
Moreover, the eigenvalue satisfies the optimal double order rate of
convergence
\[
|\omega-\omega_\ell|\le C\delta(\Ws,\Ws_\ell)^2.
\]
\label{th:main}
\end{thm}

The proof of Theorem~\ref{th:main} is based on the corresponding result
written in terms of the mixed formulations~\eqref{eq:varmixed}
and~\eqref{eq:varmixedh}.

Let $\lambda=\lambda_j$ be a simple eigenvalue of~\eqref{eq:varmixed} and
$\Wm=\Span\{(\ssigma_j,\pp_j)\}$ the associated one-dimensional eigenspace.
Let $\lambda_\ell=\lambda_{\ell,j}$ be the $j$-th discrete eigenvalue
corresponding to the $\ell$-th level of refinement in the adaptive scheme and
$\Wm_\ell=\Span\{(\ssigma_{\ell,j},\pp_{\ell,j})\}$ the associated eigenspace.
The gap between $\Wm$ and $\Wm_\ell$ is measured by
\[
\delta(\Wm,\Wm_\ell)=\sup_{\substack{(\ssigma,\pp)\in\Wm\\ \|\pp\|_0=1}}\,
\inf_{(\ssigma_\ell,\pp_\ell)\in\Wm_\ell}
\big(\|\ssigma-\ssigma_\ell\|_0^2+\|\pp-\pp_\ell\|_0^2\big)^{1/2}.
\]

We recall the reliability and efficiency properties proved in~\cite{BGRI2}. It
turns out that in the case of the mixed formulation it is possible to obtain a
local efficiency estimate.

\begin{proposition}

Let $(\lambda,\ssigma,\pp)$ and $(\lh,\sh,\pp_h)$ be  solutions of
Problems~\eqref{eq:varmixed} and~\eqref{eq:varmixedh}, respectively, such that
the latter approximates the former as $h$ goes to zero.

\begin{description}

\item[Reliability]

there exist $\rho_{\mathrm{rel1}}(h)$ and $\rho_{\mathrm{rel2}}(h)$ tending to
zero as $h\to0$ and positive constants $C$ independent of the mesh size such
that
\[
\aligned
&\|\ssigma-\sh\|_0+\|\pp-\pp_h\|_0\le
C\etam+\rho_{\mathrm{rel1}}(h)(\|\ssigma-\sh\|_0+\|\pp-\pp_h\|_0)\\
&|\lambda-\lh|\le
C\etam^2+\rho_{\mathrm{rel2}}(h)(\|\ssigma-\sh\|_0+\|\pp-\pp_h\|_0)^2.
\endaligned
\]

\item[Efficiency]
for each $K\in\tria$,
\[
\etam_K\le C(\|\ssigma-\sh\|_{0,K'}+\|\pp-\pp_h\|_{0,K'}),
\]
where $K'$ is the union of the tetrahedra sharing a face with $K$.

\end{description}

\end{proposition}

\begin{proof}

See Theorems~3 and~4 of~\cite{BGRI2}. The estimate for $|\lambda-\lh|$ is an
immediate consequence of~\eqref{eq:duran} (see next section).

\end{proof}

The counterpart of Theorem~\ref{th:main} in the framework of the mixed
formulation is stated as follows.

\begin{thm}
Provided the meshsize of the initial mesh $\T_0$ and the bulk parameter
$\theta$ are small enough, if the eigenspace satisfies
$|\Wm|_{\mathcal{A}_s}<\infty$, then the sequence of discrete eigenspaces
$\Wm_\ell$ corresponding to the solution computed on the mesh $\T_\ell$
fulfills the optimal estimate
\[
\delta(\Wm,\Wm_\ell)\le
C(\card(\T_\ell)-\card(\T_0))^{-s}|\Wm|_{\mathcal{A}_s}.
\]
Moreover, the eigenvalue satisfies the optimal double order rate of
convergence
\[
|\omega-\omega_\ell|\le C\delta(\Wm,\Wm_\ell)^2.
\]
\label{th:mainmixed}
\end{thm}

The proof of our main result has the same structure as the one
presented in~\cite{dietmar}, based on~\cite{gallistl} and~\cite{ckns}.
For this reason, we do not repeat it here, but we
conclude this section by listing some keystone properties that are essential
for the proof of our main result. We refer the interested reader
to~\cite{dietmar} and to the references therein for a rigorous proof of how to
combine them in order to get the result of Theorem~\ref{th:mainmixed}.

The following properties involve quantities related to meshes that will be
denoted by $\T_H$, $\T_h$, or $\T_\ell$.
In general, $\T_h$ denotes an arbitrary refinement of a fixed mesh $\T_H$,
while $\T_\ell$ refers to the sequence of meshes designed by the adaptive
procedure.
The eigenmode approximating
$\{\lambda,(\ssigma,\pp)\}$ will be indicated by
$\{\lambda_\kappa,(\ssigma_\kappa,\pp_\kappa)\}$ where $\kappa$ may be $H$,
$h$, or $\ell$, respectively. We assume that the sign of
$(\ssigma_\kappa,\pp_\kappa)$ is chosen in such a way that the scalar product
between $\pp$ and $\pp_\kappa$ is positive (so that the same is true for the
scalar product between $\ssigma$ and $\ssigma_\kappa$).

\begin{property}[Discrete reliability]
There exists a constant $C_{\mathrm{drel}}$ and a function
$\rho_{\mathrm{drel}}(H)$ tending to zero as $H$ goes to zero, such that, for
a sufficiently fine mesh $\T_H$ and for all refinements $\T_h$ of $\T_H$, it
holds
\[
\aligned
&\|\sh-\ssigma_H\|_0+\|\pp_h-\pp_H\|_0
\le
C_{\mathrm{drel}}\etam_H(\T_H\setminus\T_h)\\
&\qquad+\rho_{\mathrm{drel}}(H)(\|\ssigma-\sh\|_0+\|\pp-\pp_h\|_0+
\|\ssigma-\ssigma_H\|_0+\|\pp-\pp_H\|_0).
\endaligned
\]
\label{p:drel}
\end{property}

\begin{property}[Quasi-orthogonality]
There exists a function $\rho_{\mathrm{qo}}(h)$ tending to zero as $h$ goes to
zero, such that
\[
\aligned
&\errshH^2+\errphH^2\le\errsH^2+\errpH^2-\errsh^2-\errph^2\\
&\qquad+\rho_{\mathrm{qo}}(h)(\errsh^2+\errph^2+\errsH^2+\errpH^2).
\endaligned
\]
\label{p:qo}
\end{property}

\begin{property}[Contraction]
If the initial mesh $\T_0$ is sufficiently fine, there exist constants
$\beta\in(0,+\infty)$ and $\gamma\in(0,1)$ such that the term
\[
\xi^2_\ell=\etam(\T_\ell)^2+\beta(\errsl^2+\errpl^2)
\]
satisfies for all integers $\ell$
\[
\xi^2_{\ell+1}\le\gamma\xi^2_\ell.
\]
\label{p:contr}
\end{property}

In the next section we will show how to prove the above properties. While in
some cases these are natural extensions of the analogous results for the
Laplace eigenproblem in mixed form (see~\cite{dietmar}), we will see that in
particular the \emph{discrete reliability} property requires a more careful
analysis.

\section{Proof of the main results}
\label{se:auxiliary}

We start this section recalling some known results for the approximation of
problem~\eqref{eq:varmixed}. The first one is a superconvergence estimate
which has been proved in~\cite[Lemma 9]{BGRI2}.
\begin{lemma}
\label{le:superconv}
Let $(\lambda,\ssigma,\pp)$ and $(\lh,\sh,\pp_h)$ be solutions of
equations~\eqref{eq:varmixed} and~\eqref{eq:varmixedh}, respectively, with
$\|\pp\|_0=\|\pp_h\|_0=1$ and such that the latter approximates the former as
$h$ goes to zero. Then, there exists  a function $\rhosc(h)$ tending to zero as
$h\to0$ such that
\begin{equation}
\label{eq:superconv}
\|\Ph\pp-\pp_h\|_0\le\rhosc(h)(\errsh+\errph),
\end{equation}
where $\Ph$ denotes the $L^2$-projection onto $\Q_h$.
\end{lemma}

If $(\lambda,\ssigma,\pp)$ and $(\lh,\sh,\pp_h)$ are as in
Lemma~\ref{le:superconv}, thanks to the definition of $\Q_h$, it is not
difficult to verify that the following equations hold
true (see~\cite[Lemma~4]{dgp})
\begin{equation}
\label{eq:duran}
\aligned
&\lambda-\lh=\errsh^2-\lh\errph^2\\
&\lh-\lH=\errshH^2-\lH\errphH^2.
\endaligned
\end{equation}

It is useful to recall the source problem associated
with~\eqref{eq:varmixed}: given $\g\in L^2(\Omega)^3$, find
$(\sg,\pg)\in\Hocurl\times\Q$ such that
\begin{equation}
\aligned
&(\sg,\ttau)+(\curl\ttau,\pg)=0&&\forall\ttau\in\Hocurl\\
&(\curl\sg,\qq)=-(\g,\qq)&&\forall\qq\in\Q,
\endaligned
\label{eq:mixedsource}
\end{equation}
Since we have taken $\mu=1$, it turns out that $\Q=\curl(\Hocurl)=\Hodivo$,
that is the space of vectorfields in $L^2(\Omega)^3$ with zero divergence and
vanishing normal component along the boundary.

Standard regularity results for~\eqref{eq:mixedsource} imply that, if $\Omega$
is a Lipschitz polyhedron, then both components of the solution
of~\eqref{eq:mixedsource} are in $\HH^s(\Omega)$ for some $s>1/2$ (see, for
instance, the discussion related to~\cite[Theorem~2.1]{BGRI1}).

The discretization of~\eqref{eq:mixedsource} reads: find
$(\sgh,\pgh)\in\Sh\times\Q_h$ such that
\begin{equation}
\aligned
&(\sgh,\ttau)+(\curl\ttau,\pgh)=0&&\forall\ttau\in\Sh\\
&(\curl\sgh,\qq)=-(\g,\qq)&&\forall\qq\in\Q_h,
\endaligned
\label{eq:mixedsourceh}
\end{equation}
The following error estimate is well known (see~\cite{fortid})
\begin{equation}
\|\sg-\sgh\|_0+\|\pg-\pgh\|_0\le Ch^s\|\g\|_0,\quad s>1/2.
\label{eq:fortid}
\end{equation}

\subsection{Proof of Property~\ref{p:drel}}
The proof of Property~\ref{p:drel} (\emph{Discrete reliability}) constitutes
the main novelty with respect to the results present in the literature. The
structure of the proof is a combination of the analogous proof
in~\cite{dietmar} and of some of the results in~\cite{BGRI2}. However, some
new estimates are needed that will be detailed in this section. Since the
proof is composed of several steps, we summarize in Table~\ref{tb:wrap} the
structure of the proof.

\begin{table}

\begin{tcolorbox}

\[
\aligned
\nome:=&\errshH+\errphH\\
\le&\errsh+\errph+\errsH+\errpH
\endaligned
\]

\medskip

\textbf{Property~\ref{p:drel}}:
$\nome\le C\etam_H(\T_H\setminus\T_h)+\rho(H)\nome$

\medskip

\begin{itemize}

\item $\sh-\sH=\grad\alpha_h+\ww_h$\hfill\eqref{eq:helmh}

\begin{itemize}

\item $\|\grad\alpha_h\|_0\le
C\etam_H(\T_H\setminus\T_h)$\hfill\eqref{eq:tutto1}

\item $\|\ww_h\|_0^2=\textrm{I}+\textrm{II}+\textrm{III}$
\hfill\eqref{eq:trepez}

\begin{itemize}

\item $\textrm{I}\le\rho(H)\nome\|\ww_h\|_0$
\hfill\eqref{eq:tutto2.1}

\item $\textrm{II}\le CH^{1/2}\nome\errphH$
\hfill\eqref{eq:tutto2.1,5}

\item $\textrm{III}\le C\|\pp_H-\PH\pp_h\|_0\|\ww_h\|_0$
\hfill\eqref{eq:tutto3}

\end{itemize}

$\|\ww_h\|_0^2\le\rho_1(H)\nome^2+C\lH^2\|\pp_H-\PH\pp_h\|^2_0+\errphH^2$
\hfill\eqref{eq:sum3pez}

\begin{itemize}

\item $\|\pp_H-\PH\pp_h\|_0\le\|\pp_H-\hpH\|_0+\|\hpH-\PH\pp_h\|_0
=\mathrm{IV}+\mathrm{V}$

\item $\mathrm{IV}\le\rho(H)\nome+C\mathrm{V}$
\hfill\eqref{eq:pHhpH}

\item $\mathrm{V}^2\le\mathrm{A}_1\mathrm{A}_2+\mathrm{B}_1\mathrm{B}_2$
\hfill\eqref{eq:4pezzi}

\begin{itemize}

\item $\mathrm{A}_1\le\rho(H)\nome+C\mathrm{V}$
\hfill\eqref{eq:hence}

\item $\mathrm{A}_2\le C(h^s+H^s)\mathrm{V}$\hfill\eqref{eq:0}

\item $\mathrm{B}_1\le C(\errphH+\nome^2+\mathrm{V})$
\hfill\eqref{eq:one}

\item $\mathrm{B}_2\le C(h^s+H^s)\mathrm{V}$\hfill\eqref{eq:2}

\end{itemize}

$\mathrm{V}\le CH^s\nome$\hfill\eqref{eq:tutto2}

\end{itemize}

\end{itemize}
$\errshH\le C\etam_H(\T_H\setminus\T_h)+\rho(H)\nome+C\errphH$

\medskip

\item $\errphH^2=\mathrm{I}+\mathrm{II}+\mathrm{III}$
\hfill\eqref{eq:chefatica}

\begin{itemize}

\item $\mathrm{I}\le\rho(H)\nome\errphH$

\item $|\mathrm{II}+\mathrm{III}|\le C\etam_H(\T_H\setminus\T_h)\errphH$

\end{itemize}
$\errphH\le C\etam_H(\T_H\setminus\T_h)+\rho(H)\nome$\hfill\eqref{eq:p}

\end{itemize}

\end{tcolorbox}

\caption{Structure of the proof of Property~\ref{p:drel}}
\label{tb:wrap}
\end{table}

Let us start with the estimate of $\errshH$. We split $\sh-\sH$ using a
discrete Helmholtz decomposition as
\begin{equation}
\sh-\sH=\grad\alpha_h+\ww_h,
\label{eq:helmh}
\end{equation}
where $\alpha_h\in\Huo$ is a Lagrange finite element in $\N_h$ and $\ww_h$ is
an edge element in $\Sh$ satisfying
\begin{equation}
(\grad\alpha_h,\grad\psi_h)=(\sh-\sH,\grad\psi_h)\quad\forall\psi_h\in\N_h
\label{eq:defalpha}
\end{equation}
and, for some $\bbeta_h\in\Q_h$,
\begin{equation}
\aligned
&(\ww_h,\ttau)+(\curl\ttau,\bbeta_h)=0&&\forall\ttau\in\Sh\\
&(\curl\ww_h,\qq)=(\curl(\sh-\sH),\qq)&&\forall\qq\in\Q_h.
\endaligned
\label{eq:defww}
\end{equation}

In particular, $(\ww_h,\bbeta_h)$ approximates the solution of the mixed
problem~\eqref{eq:mixedsource} with source term $\g=-\curl(\sh-\sH)$.

Clearly, we have
\[
\|\grad\alpha_h\|_0\le C\|\sh-\sH\|_0,\qquad
\|\ww_h\|_{\curl}+\|\bbeta_h\|_0\le C\|\curl(\sh-\sH)\|_0.
\]

Let us estimate the first term of~\eqref{eq:helmh}.
By standard procedure, defining $\alpha_H$ as the Scott--Zhang interpolant of
$\alpha_h$ on $\T_H$ (see~\cite{ScottZhang1990}), we have
\[
\|\grad\alpha_h\|^2_0=(\grad\alpha_h,\sh-\sH)=-(\grad\alpha_h,\sH)=
-(\grad(\alpha_h-\alpha_H),\sH),
\]
since $(\grad\alpha_h,\sh)=(\grad\alpha_H,\sH)=0$ from the first equation
of~\eqref{eq:varmixedh}. Integrating by parts element by element, we get
\begin{equation}
\aligned
\|\grad\alpha_h\|^2_0&=
\sum_{K\in\T_H\setminus\T_h}\Big((\alpha_h-\alpha_H,\div\sH)-
\frac12\sum_{F\in K}\int_F(\alpha_h-\alpha_H)\jump{\sH\cdot\n}\Big)\\
&\le C\sum_{K\in\T_H\setminus\T_h}
\Big(\|\div\sH\|_{0,K}H_K\|\grad\alpha\|_{0,K}\\
&\qquad\qquad\qquad
+\frac12\sum_{F\in\F_I(K)}
\|\jump{\sH\cdot\n}\|_{0,F}H^{1/2}_F\|\alpha_h\|_{1,K}\Big)\\
&\le C\|\grad\alpha_h\|_0\Bigg(
\Big(\sum_{K\in\T_H\setminus\T_h}H^2_K\|\div\sH\|^2_{0,K}\Big)^{1/2}\\
&\qquad\qquad\qquad+\Big(\sum_{K\in\T_H\setminus\T_h}
\sum_{F\in\F_I(K)}H_F\|\jump{\sH\cdot\n}\|^2_{0,F}\Big)^{1/2}\Bigg)\\
&\le C\|\grad\alpha_h\|_0\etam_H(\T_H\setminus\T_h).
\endaligned
\label{eq:tutto1}
\end{equation}

The estimate of the second term in~\eqref{eq:helmh} requires a more careful
analysis.
\begin{equation}
\aligned
\|\ww_h\|_0^2&=(\ww_h,\ww_h)=-(\curl\ww_h,\bbeta_h)=
-(\curl(\sh-\sH),\bbeta_h)\\
&=(\lh\pp_h-\lH\pp_H,\bbeta_h)\\
&=(\lh-\lH)(\pp_h,\bbeta_h)+\lH(\pp_h-\PH\pp_h,\bbeta_h)+
\lH(\PH\pp_h-\pp_H,\bbeta_h).
\endaligned
\label{eq:trepez}
\end{equation}
We bound the three terms in the last line separately.

From classical inf-sup condition involving edge and face elements (see, for
instance,~\cite{bfgp}), we have
\[
\aligned
(\pp_h,\bbeta_h)&\le\|\pp_h\|_0\|\bbeta_h\|_0=\|\bbeta_h\|_0\\
&\le C
\sup_{\ttau_h\in\Sh}\frac{(\curl\ttau_h,\bbeta_h)}{\|\ttau_h\|_{\curl}}
=C\sup_{\ttau_h\in\Sh}\frac{(\ww_h,\ttau_h)}{\|\ttau_h\|_{\curl}}
\le C\|\ww_h\|_0.
\endaligned
\]
Hence, using the second of~\eqref{eq:duran} we conclude the estimate of the
first term in~\eqref{eq:trepez} as follows
\begin{equation}
(\lh-\lH)(\pp_h,\bbeta_h)\le
C\big(\|\sh-\sH\|_0^2+\lH\|\pp_h-\pp_H\|_0^2\big)\|\ww_h\|_0.
\label{eq:tutto2.1}
\end{equation}

We now estimate the second term in~\eqref{eq:trepez}.
Recalling the mixed problem~\eqref{eq:defww} defining $\ww_h$ and $\bbeta_h$,
we denote by $\ww_H\in\SH$ and $\bbeta_H\in\Q_H$ the corresponding solution on
the mesh $\T_H$ with $\g=-(\curl(\sh-\sH)$. Moreover, we denote by
$\ww\in\Hocurl$ and $\bbeta\in\Q$ the solution of the continuous problem
satisfying
\[
\aligned
&(\ww,\ttau)+(\curl\ttau,\bbeta)=0&&\forall\ttau\in\Hocurl\\
&(\curl\ww,\qq)=(\curl(\sh-\sH),\qq)&&\forall\qq\in\Q.
\endaligned
\]
It is clear that we have
\[
\aligned
\lH(\pp_h-\PH\pp_h,\bbeta_h)&=\lH(\pp_h-\PH\pp_h,\bbeta_h-\PH\bbeta_h)\\
&\le\lH\|\pp_h-\PH\pp_h\|_0\|\bbeta_h-\PH\bbeta_h\|_0\\
&\le\lH\|\pp_h-\pp_H\|_0\|\bbeta_h-\bbeta_H\|_0.
\endaligned
\]
The term $\|\pp_h-\pp_H\|_0$ will be estimated later. The other term can be
bounded as follows using known results for mixed problems, together with
regularity results (see the discussion after~\eqref{eq:mixedsource})
\begin{equation}
\aligned
\|\bbeta_h-\bbeta_H\|_0&\le\|\bbeta_h-\bbeta\|_0+\|\bbeta-\bbeta_H\|_0\\
&\le CH^{1/2}\|\bbeta\|_{\HH^{1/2}(\Omega)}\le CH^{1/2}\|\curl(\sh-\sH)\|_0\\
&=CH^{1/2}\|\lh\pp_h-\lH\pp_H\|_0\\
&\le CH^{1/2}\big(|\lh-\lH|+\lH\errphH\big)\\
&\le CH^{1/2}\big(\|\sh-\sH\|_0^2+\lH\errphH^2+\lH\errphH\big)\\
&\le CH^{1/2}\big((\lh+\lH)\|\sh-\sH\|_0+3\lH\errphH\big).
\endaligned
\label{eq:tutto2.1,5}
\end{equation}
The last term in~\eqref{eq:trepez} can be estimated using
again the inf-sup condition related to edge and face elements
\begin{equation}
\lH(\PH\pp_h-\pp_H,\bbeta_h)\le\lH\|\pp_H-\PH\pp_h\|_0\|\bbeta_h\|_0
\le C\lH\|\pp_H-\PH\pp_h\|_0\|\ww_h\|_0.
\label{eq:tutto3}
\end{equation}
Putting together the estimates of the three terms in~\eqref{eq:trepez}, and
using the a priori error estimate for the eigenvalue problem (see, for
instance~\cite{bfgp} or~\cite{acta}), we obtain by a suitable definition of
$\rho_1(H)$
\begin{equation}
\aligned
\|\ww_h\|^2_0&\le\rho_1(H)\big(\errshH^2+\errphH^2\big)+
C\lH^2\|\pp_H-\PH\pp_h\|^2_0+\errphH^2.
\endaligned
\label{eq:sum3pez}
\end{equation}
The estimate of $\|\pp_H-\PH\pp_h\|_0$ can be obtained in several steps. 
We consider the auxiliary problem: find $(\hsH,\hpH)\in\SH\times\Q_H$ such
that
\begin{equation}
\aligned
&(\hsH,\ttau)+(\curl\ttau,\hpH)=0&&\forall\ttau\in\SH\\
&(\curl\hsH,\qq)=-\lh(\pp_h,\qq)&&\forall\qq\in\Q_H.
\endaligned
\label{eq:hatpb}
\end{equation}
By triangular inequality we have
\[
\|\pp_H-\PH\pp_h\|_0\le\|\pp_H-\hpH\|_0+\|\hpH-\PH\pp_h\|_0.
\]
We first show that $\|\pp_H-\hpH\|_0$ is bounded by $\|\hpH-\PH\pp_h\|_0$ plus
a term which asymptotically behaves like the above estimate of
$\|\bbeta_h-\bbeta_H\|_0$.
Let $\{\lambda_{H,i},(\ssigma_{H,i},\pp_{H,i})\}$ ($i=1,\dots,N(H)$) be the
family of eigensolutions of problem~\eqref{eq:varmixedh} related to the mesh
$\T_H$ (recall that $\lH=\lambda_{H,j}$). We have
\[
\|\pp_H-\hpH\|^2_0=\sum_{i=1}^{N(H)}a_i^2,\qquad a_i=(\pp_H-\hpH,\pp_{H,i}).
\]
For $i=j$
\[
\aligned
a_j&=(\pp_H-\hpH,\pp_H)=1-(\hpH,\pp_H)=1+\frac1{\lH}(\hpH,\curl\sH)
=1-\frac1{\lH}(\hsH,\sH)\\
&=1+\frac1{\lH}(\pp_H,\curl\hsH)=1-\frac{\lh}{\lH}(\pp_h,\pp_H)
=1-\frac{\lh}{\lH}+\frac{\lh}{\lH}\big(1-(\pp_h,\pp_H)\big)\\
&=\frac{\lH-\lh}{\lH}+\frac{\lh}{2\lH}\|\pp_h-\pp_H\|^2_0
=\left(1+\frac{\lh}{2\lH}\right)\errphH^2-\frac1{\lH}\errshH^2.
\endaligned
\]
For $i\ne j$, since $a_i=-(\hpH,\pp_{H,i})$, we can proceed with the following
estimate
\[
\aligned
\lambda_{H,i}(\hpH,\pp_{H,i})&=-(\curl\ssigma_{H,i},\hpH)=
(\hsH,\ssigma_{H,i})=-(\curl\hsH,\pp_{H,i})\\
&=\lh(\pp_h,\pp_{H,i})=\lh(\PH\pp_h,\pp_{H,i}),
\endaligned
\]
which gives
\[
(\lambda_{H,i}-\lh)(\hpH,\pp_{H,i})=-\lh(\pp_{H,i},\hpH-\PH\pp_h).
\]
Hence,
\[
\aligned
\sum_{i\ne j}a_i^2&=\sum_{i\ne j}a_i\frac{\lh}{\lambda_{H,i}-\lh}
(\pp_{H,i},\hpH-\PH\pp_h)\\
&\le\max_{i\ne j}\left|\frac{\lh}{\lambda_{H,i}-\lh}\right|
\Bigg(\sum_{i\ne j}a_i^2\Bigg)^{1/2}
\Bigg(\sum_{i\ne j}(\pp_{H,i},\hpH-\PH\pp_h)^2\Bigg)^{1/2}\\
&\le\max_{i\ne j}\left|\frac{\lh}{\lambda_{H,i}-\lh}\right|
\Bigg(\sum_{i\ne j}a_i^2\Bigg)^{1/2}\|\hpH-\PH\pp_h\|_0.
\endaligned
\]
Putting things together, we get
\begin{equation}
\label{eq:pHhpH}
\aligned
&\|\pp_H-\hpH\|_0^2=\sum_{i=1}^{N(H)}a_i^2\\
&\qquad\le C(\errshH^2+\errphH^2)^2+\max_{i\ne j}
\left|\frac{\lh}{\lambda_{H,i}-\lh}\right|^2\|\hpH-\PH\pp_h\|_0^2.
\endaligned
\end{equation}
If $H$ is small enough (remember that we have assumed in Theorem~\ref{th:main}
that the initial mesh is fine enough), then the denominator
$\lambda_{H,i}-\lh$ is bounded away from zero for all $i\ne j$ and for all $h$.

From the a priori error estimate for the eigenvalue problem and a suitable
definition of $\rho_2(H)$ we get the estimate
\[
(\errshH^2+\errphH^2)^2\le\rho_2(H)(\errshH^2+\errphH^2).
\]
Now we use a duality argument in order to get a bound for
$\|\hpH-\PH\pp_h\|_0$. Let $\ppsi\in\Hocurl$ and $\pphi\in\Q$ be the solution
of
\[
\aligned
&(\ppsi,\ttau)+(\curl\ttau,\pphi)=0&&\forall\ttau\in\Hocurl\\
&(\curl\ppsi,\qq)=(\hpH-\PH\pp_h,\qq)&&\forall\qq\in\Q
\endaligned
\]
and $\ppsi_h\in\Sh$ and $\pphi_h\in\Q_h$ the corresponding discrete solution on
the mesh $\T_h$. We observe that $(\ppsi,\pphi)$ is the solution
of~\eqref{eq:mixedsource} when $\g=-(\hpH-\PH\pp_h)$, hence both the components
of the solution belong to $\HH^s(\Omega)$ for some $s>1/2$.
Let $\fortinH$ be the Fortin operator introduced
in~\cite{fortid} associated with problem~\eqref{eq:mixedsource}, with the
following properties:
\begin{equation}
\label{eq:Fortid}
\aligned
&\fortinH:\HH^s(\Omega)\to\SH\\
&(\curl(\ppsi-\fortinH\ppsi),\qq)=0
\qquad\forall\qq\in\Q_H,\ \forall\ppsi\in\HH^s(\Omega)\\
&\|\fortinH\ppsi\|_{\curl}\le C\|\ppsi\|_\HH^s(\Omega)\\
&\|\ppsi-\fortinH\ppsi\|_0\le C\|\ppsi\|_\HH^s(\Omega).
\endaligned
\end{equation}
We have
\[
\aligned
&\|\hpH-\PH\pp_h\|_0^2=(\curl\ppsi_h,\hpH-\PH\pp_h)&&\\
&\quad=(\curl\fortinH\ppsi_h,\hpH-\PH\pp_h)&&\text{def.\ of $\fortinH$}\\
&\quad=(\curl\fortinH\ppsi_h,\hpH-\pp_h)&&\\
&\quad=-(\hsH-\sh,\fortinH\ppsi_h)&&
\text{err.\ eq.~\eqref{eq:hatpb}-\eqref{eq:varmixedh}}\\
&\quad=-(\hsH-\sh,\fortinH\ppsi_h-\ppsi_h)-(\hsH-\sh,\ppsi_h)&&\\
&\quad=-(\hsH-\sh,\fortinH\ppsi_h-\ppsi_h)+(\curl(\hsH-\sh),\pphi_h)&&
\text{duality arg.}\\
&\quad=
-(\hsH-\sh,\fortinH\ppsi_h-\ppsi_h)+(\curl(\hsH-\sh),\pphi_h-\PH\pphi_h).&&
\text{err.\ eq.~\eqref{eq:hatpb}-\eqref{eq:varmixedh}}\\
\endaligned
\]
By Cauchy--Schwarz inequality, we obtain
\begin{equation}
\label{eq:4pezzi}
\aligned
\|\hpH-\PH\pp_h\|_0^2&\le\|\hsH-\sh\|_0\|\fortinH\ppsi_h-\ppsi_h\|_0\\
&\quad+\|\curl(\hsH-\sh)\|\|\pphi_h-\PH\pphi_h\|_0,
\endaligned
\end{equation}
and we estimate separately the four norms on the right hand side.
The triangular inequality gives
\[
\|\hsH-\sh\|_0\le\|\hsH-\sH\|_0+\|\sH-\sh\|_0.
\]
We only need to estimate the first term, for which we proceed as before by
expanding it in terms of the eigensolutions on the mesh $\T_H$. Since
$\|\ssigma_{H,i}\|_0^2=\lambda_{H,i}$, we set
$\tHi=\ssigma_{H,i}/\sqrt{\lambda_{H,i}}$. We have
\[
\|\hsH-\sH\|_0^2=\sum_{i=1}^{N(H)}b_i^2,\qquad b_i=(\hsH-\sH,\tHi).
\]
For $i=j$
\[
\aligned
b_j&=(\hsH-\sH,\tH)=(\hsH,\tH)-\sqrt{\lH}=-(\curl\tH,\hpH)-\sqrt{\lH}\\
&=\lH\Bigg(\frac{\pp_H}{\sqrt{\lH}},\hpH\Bigg)-\sqrt{\lH}
=\sqrt{\lH}\big((\pp_H,\hpH)-1\big)\\
&=-\sqrt{\lH}\Bigg(\frac{\lH-\lh}{\lH}+\frac{\lh}{2\lH}\errphH^2\Bigg).
\endaligned
\]
Using~\eqref{eq:duran} we get
\[
b_j\le C\big(\errshH^2+\errphH^2\big).
\]
For $i\ne j$
\[
\aligned
b_i&=(\hsH,\tHi)=-(\curl\tHi,\hpH)=\sqrt{\lambda_{H,i}}(\pp_{H,i},\hpH)\\
&=\sqrt{\lambda_{H,i}}\frac{\lh}{\lh-\lambda_{H,i}}(\pp_{H,i},\hpH-\PH\pp_h),
\endaligned
\]
where in the last two estimates we took advantage of the already computed
bounds for $\big((\pp_H,\hpH)-1\big)$ and $(\pp_{H,i},\hpH)$. Hence
\[
\aligned
\sum_{i\ne j}b_i^2&=
\sum_{i\ne j}b_i\frac{\sqrt{\lambda_{H,i}}\lh}{\lh-\lambda_{H,i}}
(\pp_{H,i},\hpH-\PH\pp_h)\\
&=\max_{i\ne j}\frac{\sqrt{\lambda_{H,i}}\lh}{|\lh-\lambda_{H,i}|}
\Big(\sum_{i\ne j}b_i^2\Big)^{1/2}\|\hpH-\PH\pp_h\|_0.
\endaligned
\]
As we have already observed, the denominator of the last expression is always
bounded away from zero for $H$ small enough.
It follows that
\[
\aligned
\Big(\sum_{i\ne j}b_i^2\Big)^{1/2}&\le
\max_{i\ne j}\frac{\sqrt{\lambda_{H,i}}\lh}{|\lh-\lambda_{H,i}|}
\|\hpH-\PH\pp_h\|_0\\
&\le\max_{i\ne j}\Big(\lh/\sqrt{\lambda_{H,i}}\Big)
\frac{\lambda_{H,i}/\lh}{|1-\lambda_{H,i}/\lh|}\|\hpH-\PH\pp_h\|_0\\
&\le C\|\hpH-\PH\pp_h\|_0.
\endaligned
\]
Hence,
\begin{equation}
\aligned
\|\hsH-\sH\|_0^2&\le C\big(\errshH^4+\errphH^4+\|\hpH-\PH\pp_h\|_0^2\big)\\
&\le\rho_3(H)\big(\errshH^2+\errphH^2\big)+C\|\hpH-\PH\pp_h\|_0^2.
\endaligned
\label{eq:hence}
\end{equation}
To bound the second term in~\eqref{eq:4pezzi}, we use the triangle inequality,
the error estimates for the mixed source problem~\eqref{eq:fortid},
the properties of the Fortin operator~\eqref{eq:Fortid}
\begin{equation}
\aligned
\|\fortinH\ppsi_h-\ppsi_h\|_0&\le
\|\fortinH(\ppsi_h-\ppsi)\|_0+\|\fortinH\ppsi-\ppsi\|_0
+\|\ppsi-\ppsi_h\|_0\\
&\le C\|\ppsi-\ppsi_h\|_0+\|\fortinH\ppsi-\ppsi\|_0\\
&\le C (h^s+H^s)\|\hpH-\PH\pp_h\|_0.
\endaligned
\label{eq:0}
\end{equation}
From the definition of the discrete spaces, since $\Q_h=\curl(\Sh)$ for any
choice of the mesh, it is clear that
\[
\curl(\hsH)=-\lh\PH\pp_h,\qquad \curl(\sh)=-\lh\pp_h.
\]
Therefore, from~\eqref{eq:pHhpH} we obtain
\begin{equation}
\aligned
\|\curl(\sh-\hsH)\|_0&\le\lh\|\pp_h-\PH\pp_h\|_0
\le\lh(\errphH+\|\pp_H-\PH\pp_h\|_0)\\
&\le\lh\big(\errphH+C(\errshH^2+\errphH^2)\\
&\quad+C\|\hpH-\PH\pp_h\|_0\big).
\endaligned
\label{eq:one}
\end{equation}
Considering again the definition of the solution of the dual problem,
the last term in~\eqref{eq:4pezzi} can be bounded by using~\eqref{eq:fortid}
and the properties of the projection operator $\PH$:
\begin{equation}
\aligned
\|\pphi_h-\PH\pphi_h\|_0&\le\|\pphi_h-\pphi\|_0+\|\pphi-\PH\pphi\|_0
+\|\PH(\pphi-\pphi_h)\|_0\\
&\le C\|\pphi-\pphi_h\|_0+\|\pphi-\PH\pphi\|_0\\
&\le C(h^s+H^s)\|\hpH-\PH\pp_h\|_0.
\endaligned
\label{eq:2}
\end{equation}
Putting together the estimates of the four norms in~\eqref{eq:4pezzi}, we
arrive at
\[
\aligned
\|\hpH-\PH\pp_h\|_0^2&\le CH^s\|\hpH-\PH\pp_h\|_0
\big(\errshH+\errphH\big)\\
&\quad+C(h^s+H^s)\|\hpH-\PH\pp_h\|_0^2,
\endaligned
\]
which implies that, for $H$ sufficiently small, we have
\begin{equation}
\|\hpH-\PH\pp_h\|_0\le CH^s\big(\errshH+\errphH\big).
\label{eq:tutto2.3}
\end{equation}
It turns out that the final estimate for~\eqref{eq:trepez} is obtained
from~\eqref{eq:tutto2.1}, \eqref{eq:tutto2.1,5}, and~\eqref{eq:tutto2.3} as
follows:
\begin{equation}
\|\ww_h\|_0\le\rho_4(H)(\errshH+\errphH)+C\errphH
\label{eq:tutto2}
\end{equation}
with an appropriate definition of $\rho_4(H)$.

Finally, from~\eqref{eq:helmh}, \eqref{eq:tutto1}, and~\eqref{eq:tutto2}, we
have
\begin{equation}
\errshH\le C\etam_H(\T_H\setminus\T_h)+\rho_4(H)(\errshH+\errphH)+C\errphH.
\label{eq:sigma}
\end{equation}

We now move to the term $\|\pp_h-\pp_H\|_0$. We consider the following
auxiliary problem: find $\cchi_h\in\Sh$ and $\zz_h\in\Q_h$ such that
\begin{equation}
\aligned
&(\cchi_h,\ttau)+(\curl\ttau,\zz_h)=0&&\forall\ttau\in\Sh\\
&(\curl\cchi_h,\qq)=(\pp_h-\pp_H,\qq)&&\forall\qq\in\Q_h.
\endaligned
\label{eq:1}
\end{equation}
Therefore, we have $\curl\cchi_h=\pp_h-\pp_H$ and
$\|\cchi_h\|_{\curl}\le C\|\pp_h-\pp_H\|_0$.

We are going to use a technical tool introduced in~\cite[Theorem~4.1]{zcswx}.
More precisely, if $\T_h$ is a refinement of $\T_H$, there exists an operator
$\hx:\Sh\to\SH$ such that for all $\ttau\in\Sh$ it holds $\hx\ttau=\ttau$ on
the elements of $\T_H$ that have not been refined (more precisely, on the
elements of $\T_H$ whose closures have no intersection with the closures of
any refined elements). Such operator is stable in the $\H(\curl)$-norm, i.e.,
$\|\hx\ttau\|_{\curl}\le C\|\ttau\|_{\curl}$ for all $\ttau\in\Sh$.

We get
\begin{equation}
\aligned
\errphH^2&=(\pp_h-\pp_H,\curl\cchi_h)=-(\sh,\cchi_h)-(\pp_H,\curl\cchi_h)\\
&=-(\sh,\cchi_h)-(\pp_H,\curl(\cchi_h-\hx\cchi_h)-(\pp_H,\curl\hx\cchi_h)\\
&=-(\sh,\cchi_h)-(\pp_H,\curl(\cchi_h-\hx\cchi_h)+(\sH,\hx\cchi_h)\\
&=-(\sh-\sH,\cchi_h)-(\pp_H,\curl(\cchi_h-\hx\cchi_h)-(\sH,\cchi_h-\hx\cchi_h)
\endaligned
\label{eq:schoberl}
\end{equation}

Let us set $\xxi_h=\cchi_h-\hx\cchi_h$ and denote by $\schoberl$ the operator
introduced in~\cite[Theorem~1]{schoberl} mapping $\Hocurl$ into the space of
lowest order N\'ed\'elec elements so that there exist $\varphi\in\Huo$ and
$\ss\in\mathbf{H}^1_0(\Omega)$ satisfying
\[
\aligned
&\xxi_h-\schoberl\xxi_h=\grad\varphi+\ss\\
&h^{-1}_K\|\varphi\|_{0,K}+\|\grad\varphi\|_{0,K}\le C\|\xxi_h\|_{0,K'}\\
&h^{-1}_K\|\ss\|_{0,K}+\|\grad\ss\|_{0,K}\le C\|\curl\xxi_h\|_{0,K'}
\endaligned
\]
for all $K\in\T_h$ and with $K'$ denoting the union of elements in $\T_h$
sharing at least a vertex with $K$.

From the first equation in~\eqref{eq:varmixedh} is follows that
$(\sH,\schoberl\xxi_h)+(\curl\schoberl\xxi_h,\pp_H)=0$. This implies
that~\eqref{eq:schoberl} gives
\begin{equation}
\errphH^2=-(\sh-\sH,\cchi_h)-(\pp_H,\curl(\xxi_h-\schoberl\xxi_h))
-(\sH,\xxi_h-\schoberl\xxi_h).
\label{eq:chefatica}
\end{equation}

The first term can be estimated as follows.
\[
\aligned
-(\sh-\sH,\cchi_h)&=-(\sh-\hsH,\cchi_h)-(\hsH-\sH,\cchi_h)&&\\
&=(\curl(\sh-\hsH),\zz_h)-(\hsH-\sH,\cchi_h)&&\text{Eq.~\eqref{eq:1}}\\
&=-(\hsH-\sH,\cchi_h)&&\text{def.\ of $\hsH$}\\
&\le\|\hsH-\sH\|_0\|\cchi_h\|_0&&\\
&\le\|\hsH-\sH\|_0\errphH.&&
\endaligned
\]
The estimate for $\|\hsH-\sH\|_0$ follows from~\eqref{eq:hence}
and~\eqref{eq:tutto2.3} and is given by
\[
\|\hsH-\sH\|_0\le\rho_5(H)(\errshH+\errphH).
\]
The remaining two terms in~\eqref{eq:chefatica} can be bounded together.
\[
\aligned
(&\pp_H,\curl(\xxi_h-\schoberl\xxi_h))+(\sH,\xxi_h-\schoberl\xxi_h)\\
&=\sum_{K\in\T_H\setminus\T_h}\Big(\int_K\curl\pp_H\cdot\ss+\frac12
\sum_{F\in\F_I(K)}\int_F\jump{\pp_\H\times\n}\cdot\ss\Big)
+(\sH,\ss)+(\sH,\grad\varphi)\\
&=\sum_{K\in\T_H\setminus\T_h}\Big(\int_K(\sH+\curl\pp_H)\cdot\ss+\frac12
\sum_{F\in\F_I(K)}\int_F\jump{\pp_\H\times\n}\cdot\ss\Big)\\
&\quad+\sum_{K\in\T_H\setminus\T_h}\Big(-\int_K\div\sH\varphi+\frac12
\sum_{F\in\F_I(K)}\int_F\jump{\sH\cdot\n}\varphi\Big).
\endaligned
\]
Therefore,
\[
\aligned
\big|(\pp_H&,\curl(\xxi_h-\schoberl\xxi_h))+
(\sH,\xxi_h-\schoberl\xxi_h)\big|\\
&\le\sum_{K\in\T_H\setminus\T_h}\Big(\|\sH+\curl\pp_H\|_{0,K}\|\ss\|_{0,K}
+\frac12\sum_{F\in\F_I(K)}\|\jump{\pp_\H\times\n}\|_{0,F}\|\ss\|_{0,F}\Big)\\
&\quad+\sum_{K\in\T_H\setminus\T_h}\Big(\|\div\sH\|_{0,K}\|\varphi\|_{0,K}
+\frac12\sum_{F\in\F_I(K)}\|\jump{\sH\cdot\n}\|_{0,F}\|\varphi\|_{0,F}\Big)\\
&\le C\sum_{K\in\T_H\setminus\T_h}
\Big(H_K\|\sH+\curl\pp_H\|_{0,K}\|\curl\xxi_h\|_{0,K'}\\
&\quad+\frac12\sum_{F\in\F_I(K)}
H_F^{1/2}\|\jump{\pp_\H\times\n}\|_{0,F}\|\curl\xxi_h\|_{0,K'}\Big)\\
&\quad+C\sum_{K\in\T_H\setminus\T_h}
\Big(H_K\|\div\sH\|_{0,K}\|\xxi_h\|_{0,K'}\\
&\quad+\frac12\sum_{F\in\F_I(K)}
H_F^{1/2}\|\jump{\sH\cdot\n}\|_{0,F}\|\xxi_h\|_{0,K'}\Big)\\
&\quad\le C\etam_H(\T_H\setminus\T_h)\|\xxi_h\|_{\curl}
\le C\etam_H(\T_H\setminus\T_h)\errphH.
\endaligned
\]

Finally, Equation~\eqref{eq:chefatica} becomes
\begin{equation}
\errphH\le C\etam_H(\T_H\setminus\T_h)+\rho_5(H)(\errshH+\errphH).
\label{eq:p}
\end{equation}

Putting things together, estimates~\eqref{eq:sigma} and~\eqref{eq:p} give the
final result.

\subsection{Proof of Property~\ref{p:qo}}
The proof of Property~\ref{p:qo} (\emph{Quasi-orthogonality}) can be obtained
after appropriate modification of the analogous result in~\cite{dietmar}.

By direct computation we have
\[
\aligned
&\errshH^2=\errsH^2-\errsh^2-2(\ssigma-\sh,\sh-\sH)\\
&\errphH^2=\errpH^2-\errph^2-2(\Ph\pp-\pp_h,\pp_h-\pp_H).
\endaligned
\]
Since $\T_h$ is a refinement of $\T_H$, we have that $\sH\in\Sh$, hence the
error equations relative to~\eqref{eq:varmixed} and~\eqref{eq:varmixedh} give
\[
\aligned
(\ssigma-\sh,\sh-\sH)&=-(\curl(\sh-\sH),\pp-\pp_h)\\
&=(\lh\pp_h-\lH\pp_H,\pp-\pp_h)\\
&=(\lh\pp_h-\lH\pp_H,\Ph\pp-\pp_h).
\endaligned
\]
Using Lemma~\ref{le:superconv} and the equalities in~\eqref{eq:duran}, we
obtain
\[
\aligned
&(\ssigma-\sh,\sh-\sH)+(\Ph\pp-\pp_h,\pp_h-\pp_H)\\
&\quad=(\lh\pp_h-\lH\pp_H,\Ph\pp-\pp_h)+(\pp_h-\pp_H,\Ph\pp-\pp_h)\\
&\quad\le\big(|\lh-\lH|+(1+\lH)\errphH\big)\|\Ph\pp-\pp_h\|_0\\
&\quad\le\big(\errshH^2+\lH\errphH^2+(1+\lH)\errphH\big)\\
&\quad\qquad\rhosc(h)\big(\errsh+\errph\big)
\endaligned
\]
which, using Young inequality, gives the desired result.

\subsection{Proof of Property~\ref{p:contr}}
The contraction property is quite standard in the framework of adaptive
schemes, see~\cite{ckns}.
It is a consequence of the following error estimator reduction property: there
exist constants $\beta_1\in(0,+\infty)$ and $\gamma_1\in(0,1)$ such that, if
$\T_{\ell+1}$ is the refinement of $\T_\ell$ generated by the adaptive scheme,
it holds
\[
\etam(\T_{\ell+1})^2\le\gamma_1\etam(\T_\ell)^2+
\beta_1\big(\errsl^2+\errpl^2\big).
\]
In our case, the proof can be obtained with natural
modifications from the one outlined in~\cite{dietmar} and using the following
notation:
\[
e_\ell^2=\|\ssigma-\ssigma_\ell\|^2_0+\|\pp-\pp_\ell\|^2_0,\qquad
\mu_\ell^2=\etam(\T_\ell)^2.
\]

\section{Conclusions}

In this paper we have proved the optimal convergence of an
adaptive finite element scheme for the approximation of the eigensolutions of
the Maxwell system. The scheme makes use of N\'ed\'elec edge finite element in
three space dimensions and a standard residual-based error indicator. The
proof is based on an equivalent mixed formulation. The most challenging part
of the proof consists in showing a suitable discrete reliability property.

\section*{Acknowledgments}

The authors are members of the INdAM Research group GNCS. This work has been
partially supported by IMATI/CNR.

\bibliographystyle{plain}
\bibliography{ref}

\end{document}